\documentclass[12pt]{article}

\usepackage{amsthm,amsmath,amssymb}

\newtheorem{formula}{}[section]

\newtheorem{definition}[formula]{Definition}
\newtheorem{corollary}[formula]{Corollary}
\newtheorem{lemma}[formula]{Lemma}
\newtheorem{theorem}[formula]{Theorem}

\theoremstyle{definition}
\newtheorem{example}[formula]{Example}
\newtheorem{remark}[formula]{Remark}
\newtheorem*{notation}{Notation}

\def\thrm{\begin{theorem}}
\def\thrml#1{\begin{theorem}\label{#1}}
\def\ethrm{\end{theorem}}
\def\dfntn{\begin{definition}}
\def\dfntnl#1{\begin{definition}\label{#1}}
\def\edfntn{\end{definition}}
\def\nmrt{\begin{enumerate}}
\def\enmrt{\end{enumerate}}
\def\tm#1{\item[{\rm (#1)}]}
\def\qtn{\begin{equation}}
\def\qtnl#1{\begin{equation}\label{#1}}
\def\eqtn{\end{equation}}
\def\lmm{\begin{lemma}}
\def\lmml#1{\begin{lemma}\label{#1}}
\def\elmm{\end{lemma}}
\def\crllr{\begin{corollary}}
\def\crllrl#1{\begin{corollary}\label{#1}}
\def\ecrllr{\end{corollary}}
\def\rmrk{\begin{remark}}
\def\rmrkl#1{\begin{remark}\label{#1}}
\def\ermrk{\end{remark}}
\def\css{\begin{cases}}
\def\ecss{\end{cases}}

\newcommand{\C}{\mathbb{C}}
\newcommand{\F}{\mathbb{F}}

\newcommand{\prm}{{\cal P}}
\newcommand{\Q}{\mathbb{Q}}

\newcommand{\Z}{\mathbb{Z}}

\DeclareMathOperator{\Fr}{Fr}

\DeclareMathOperator{\im}{Im}

\DeclareMathOperator{\Mat}{Mat}
\DeclareMathOperator{\rk}{rk}
\DeclareMathOperator{\sym}{Sym}

\def\lg{\langle}
\def\ov{\overline}
\def\rg{\rangle}

\def\myskip{\medskip}

\begin{document}

\title{A modular absolute bound condition\\ for primitive association schemes}
\author{
Akihide Hanaki\\
\small Faculty of Science, Shinshu University\\
\small Matsumoto, 390-8621, Japan\\
\small E-mail address: hanaki@math.shinshu-u.ac.jp
\and
Ilia Ponomarenko
\thanks{Partially supported by RFBR grants 05-01-00899, 07-01-00485 and NSH-4329.2006.1}\\
\small Petersburg Department of V.A.Steklov\\
\small Institute of Mathematics, 191023, Russia\\
\small E-mail address: inp@pdmi.ras.ru
}
\date{September 28, 2007}
\maketitle

\begin{abstract}
The well-known absolute bound condition for a primitive symmetric association scheme $(X,S)$
gives an upper bound for $|X|$ in terms of $|S|$ and the minimal non-principal
multiplicity of the scheme. In this paper we prove another upper bounds for $|X|$ for an
arbitrary primitive scheme $(X,S)$. They do not depend on $|S|$ but depend on some invariants
of its adjacency algebra $KS$ where $K$ is an algebraic number field or a finite field.
\end{abstract}

\section{Introduction}
Let $(X,S)$ be an association scheme (for a background on association scheme theory
we refer to \cite{BI84,Z05} and Appendix).
Denote by $F S$ its {\em adjacency algebra}  over a field~$F$. As usual we
consider $F S$ as a subalgebra of the full matrix algebra $\Mat_X(F)$. Set
$$
\rk_{\min}(F,S)=\min_{A\in F S\setminus F J}\rk(A)
$$
where $J$ is the all-one matrix in $F S$ and $\rk(A)$ is the rank of a matrix $A$.
One can see that in the commutative case the number $\rk_{\min}(\C,S)$ coincides with
the minimal multiplicity $m_{\min}$ of a non-principal irreducible representation of
the algebra $\C S$ (see~(\ref{150807a})).
\myskip

It is a well-known fact (see \cite[Theorem~4.9]{BI84}) that
given a primitive symmetric scheme $(X,S)$ the number $|X|$ can not be arbitrarily
large when $|S|$ and $m_{\min}$ are bounded. It was asked there about a
reasonable absolute bound condition for an arbitrary primitive commutative scheme. The main goal
of this paper is to use the modular representation theory for schemes to get another upper bound
for $|X|$ without the assumption of commutativity. Our first result gives the following
{\em modular absolute bound condition} for primitive schemes.

\begin{theorem}\label{180707b}
Let $(X,S)$ be a primitive scheme
and let $q$ be a prime power.
Set $r=\rk_{\min}(\F_q,S)$.
Then
$$
|X|\le\frac{q^{r}-1}{q-1}
$$
whenever $r>1$.
If $r=1$, then $|X|<q$ and $(X,S)$ is a thin scheme of prime order.
\end{theorem}

We have examples for which the equality holds in Theorem \ref{180707b}.

\begin{example}
Let $(X,S)$ be the cyclotomic scheme over a prime field $\F_p$ corresponding to
its multiplicative subgroup of order $r$.
Suppose that there exists a prime $q$ such that $p=(q^r-1)/(q-1)$.
Then $\rk_{\min}(\F_q,S)=r$ and the equality in Theorem \ref{180707b} holds.
We omit the proof of this fact, but one can easily check it for
$(p,r,q)=(31,5,2)$ or $(31,3,5)$.
\end{example}

Given  a scheme $(X,S)$ denote by $\prm=\prm(\C S)$ the set of all
central primitive idempotents of the algebra $\C S$.
For $P\in\prm$ set $m_P$ to be the multiplicity of the irreducible
representation of $\C S$ corresponding to $P$ in the standard
representation (in $\Mat_X(\C)$).
Put
\qtnl{150807a}
m_{\min}=\min_{P\in\prm\setminus\{P_0\}}m_P
\eqtn
where $P_0=(1/|X|)J$ is the principal idempotent of~$\C S$.
If $(X,S)$ is primitive and $m_{\min}=1$,
then it is a thin scheme of prime order (Theorem~\ref{200707b}).
So we may avoid this case. Set
$$
\Q(X,S)=\Q(\{P_{x,y}:\, P\in\prm,\ x,y\in X\}).
$$
It is not necessary a splitting field of $\Q S$, but it is a Galois extension and every
$P\in\prm$ belongs to the adjacency algebra $\Q(X,S)S$.

\begin{theorem}\label{160707a}
Let $(X,S)$ be a primitive scheme. Suppose that $p$ is a prime which does not
divide the Frame number of this scheme and $m_{\min}>1$. Then
\qtnl{170707b}
|X|\le \frac{q^{m_{\min}}-1}{q-1}.
\eqtn
where $q=p^{|\Q(X,S):\Q|}$.
\end{theorem}

\begin{remark}
The proof of Theorem~\ref{160707a} given in Section~\ref{180707c} shows that the upper bound
in (\ref{170707b}) can be reduced.
For an appropriate $P\in\prm$,
the bound is
$$|X|\le \frac{q^{m_P}-1}{q-1}$$
where $q=|\Q(\{P_{x,y}:\, x,y\in X\}):\Q|$.
\end{remark}

We do not know any primitive scheme for which the upper bound in (\ref{170707b}) is tight.
However, there are examples where this bound is less than one given by the absolute bound condition
(e.g. for some amorphic primitive schemes $(X,S)$ such that $|X|=q^2$ and $|S|=(q+1)/2$
where $q$ is a prime and $q\not\equiv 1\,\pmod{3}$; in this case $\Q(X,S)=\Q$ and one can take $p=3$).
On the other hand, one can use inequality (\ref{170707b}) to prove the finiteness of some
classes of {\it rational} primitive schemes (here a scheme is called rational if $\Q$ is a
splitting field of its adjacency algebra). In this case, we can apply
Theorem~\ref{160707a} to any primitive {\it $p'$-scheme}. By the definition for such a scheme
the prime $p$ does not divide neither $|X|$ nor the valency of an element from $S$.

\begin{corollary}\label{170707a}
Given a prime $p$ and a positive integer $r$ the set of rational primitive $p'$-schemes
for which $m_{\min}\le r$, is finite.
\end{corollary}

The class of $p'$-schemes with $p=2$ consists of {\it odd} schemes, i.e. those for which
only symmetric basis relation of it is a reflexive one.
Theorem~\ref{160707a} shows that for a fixed $r$ a splitting field of an odd primitive
scheme with $m_{\min}\le r$ grows when $|X|$ grows.
\myskip

The assumption in Theorem~\ref{160707a} requires that the adjacency algebra
of $(X,S)$ over a field of characteristic $p$ is semisimple.
Non-semisimple case seems to be
much more difficult (see~\cite{HY05,P02}).
\myskip

The proofs of Theorems~\ref{180707b}
and \ref{160707a} are given in Sections~\ref{220707a} and \ref{180707c} respectively.
To make the paper self-contained
we put in Appendix the notations and definitions concerning schemes
and their adjacency algebras.

\begin{notation}
As usual by $\Q$, $\C$ and $\F_q$ we denote the fields of rational and complex numbers
and a finite field with $q$ elements respectively.
Throughout the paper $X$ denotes a finite set. The diagonal of the set $X^2$ is denoted
by~$\Delta$. 
The algebra of all matrices whose entries belong to a field $F$ and whose rows and columns
are indexed by the elements of $X$ is denoted by $\Mat_X(F)$, the identity
matrix by $I$ and the all-one matrix by~$J.$
Given $A\in\Mat_X(F)$ and $x,y\in X$, we denote by $A_{x,y}$ the $(x,y)$-entry of $A$.
The Hadamard (componentwise) product of matrices $A,B\in\Mat_X(F)$ is denoted by $A\circ B$.
The adjacency matrix of a binary relation $r\subset X^2$ is denoted by $A_r$ (this is a
\{0,1\}-matrix of $\Mat_X(F)$ such that $(A_r)_{x,y}=1$ if oand only if $(x,y)\in r$).
The left standard module of the algebra $\Mat_X(F)$ is denoted by $F X$. We will identify
$X$ with a subset of $F X$.
\end{notation}


\section{Combinatorics in the adjacency algebra}\label{220707a}

First we prove that with any matrix of the adjacency algebra of a scheme one can
associate some special relations which are unions of basis relations (a
special case of our result also follows from \cite[Lemma~4.1]{EP97}). Namely, let $F$ be
a field. Given a matrix $A\in\Mat_X(F)$ and an element $\lambda\in F$ we define a binary
relation
$$
e_\lambda(A)=\{(x,y)\in X\times X:\ \lambda Ax=Ay\}
$$
on the set $X$. Clearly, $e_1(A)$ is a nonempty equivalence relation on $X$ and
$e_\lambda(A)\cap e_\mu(A)=\emptyset$ for all nonzero elements $\lambda\ne\mu$.
Besides, $e_0(A)=\emptyset$ if and only if the matrix $A$ has no zero columns.
In the latter case, the relation
\qtnl{150807e}
e(A)=\bigcup_{\lambda\in F}e_{\lambda}(A)
\eqtn
is also an equivalence relation on $X$. Note that $Ax$ being the $x$th column of the matrix $A$
can be considered as an element of $FX$. So $(x,y)\in e(A)$ if and only if the vectors
$Ax,Ay\in FX$ are linearly dependent.
\myskip

In \cite[Lemma~4.1]{EP97} it was proved that given a scheme $(X,S)$ and a matrix $A\in\C S$
the relation $e_\lambda(A)$ with $\lambda=1$ belongs to the set $S^*$ of all unions of
relations from~$S$. The following statement generalizes this result for an arbitrary field
and all $\lambda$'s. Below we denote by $A_e$ and $A_\lambda$ the adjacency matrices of the
relations $e(A)$ and $e_{\lambda}(A)$ respectively. In the first part of the proof we
follow to \cite[Lemma~1.42]{E05}.

\begin{theorem}\label{110707c}
Let $(X,S)$ be a scheme and let $F$ be a field. Then $e_\lambda(A)\in S^*$ for all $A\in FS$
and $\lambda\in F$.
\end{theorem}

\begin{proof}
Without loss of generality we assume that $A\ne 0$. First we suppose that $F=\C$. Since
$A\in\C S$, we also have $A^*\in\C S$ where $A^*$ is the Hermitian conjugate of~$A$. This
implies that $A^*A\in\C S$. So given $x\in X$ the number $(A^*A)_{x,x}$ equals to the
coefficient of the identity matrix $I=A_\Delta$ in the decomposition of the matrix $A^*A$ by the matrices $A_s$, $s\in S$.
Denote it by~$d$. Then by the Cauchy-Schwartz inequality we conclude that
\qtnl{190707a}
|(A^*A)_{x,y}|=|\lg Ax,Ay\rg|\le\|Ax\|\cdot\|Ay\|=d
\eqtn
where $\lg\ \cdot\ ,\ \cdot\ \rg$ and $\|\cdot\|$ are the inner product and the Euclidean norm
in $\C X$ respectively.
Moreover, the equality
in (\ref{190707a}) is attained if and only if the vectors $Ax$ and $Ay$ are linearly
dependent. Thus $|(A^*A)_{x,y}|=d$ if and only if $(x,y)\in e(A)$.
Due to (\ref{240707b}) this shows that $A_e\in\C S$ and so $e(A)\in S^*$.
On the other hand, given $(x,y)\in e_\lambda(A)$ the number
$$
(A^*A)_{x,y}=\lg Ax,Ay\rg=\lg Ax,\lambda Ax\rg=\ov\lambda\lg Ax,Ax\rg=\ov\lambda d.
$$
does not depend on $(x,y)$. By (\ref{150807e}) this means that
$$
(A^*A)\circ A_e=d\sum_{\lambda\in\Lambda}\ov\lambda A_\lambda
$$
where $\Lambda=\{\lambda\in\F:\ e_\lambda(A)\ne\emptyset\}$. Since the matrices $A^*A$
and $A_e$ belong to $\C S$,
we conclude by (\ref{240707b}) that $A_\lambda\in\C S$ and hence $e_\lambda(A)\in S^*$.
\myskip

Let $F$ be an arbitrary field. Since $A\in FS$, any two columns of~$A$ consist of the
same elements of~$F$. Denote the set of all of them by~$M$. Then
\qtnl{150807f}
M\lambda=M,\quad\lambda\in\Lambda,
\eqtn
where $\Lambda$ is as above.
Easily we can see that $\Lambda$ is a finite subgroup of the multiplicative group $F^\times$,
and so $\Lambda$ is cyclic.
Take an injection and a group monomorphism
$$
f:M\to\C,\ \mu\mapsto\mu',\qquad\varphi:\Lambda\to\C^\times,\ \lambda\mapsto\lambda'
$$
such that the permutation groups induced by the actions of $\Lambda$ on $M$, and
of $\Lambda'=\im(\varphi)$ on $M'=\im(f)$ are equivalent. Then it is easy to see that
$$
\lambda Ax=Ay\ \Leftrightarrow\ \lambda' A'x=A'y,\qquad x,y\in X,
$$
where $A'\in\Mat_X(\C)$ is the complex matrix with entries $A'_{x,y}=(A_{x,y})^f$
for all~$x,y$. So $e(A)=e(A')$ and we are done by the first part of the proof.
\end{proof}

It was proved in \cite[p.71]{W76} that any primitive scheme having a nonreflexive
basis relation of valency~$1$ is a thin scheme of prime order. The
following theorem gives a ``dual'' version of this result.

\begin{theorem}\label{200707b}
Let $(X,S)$ be a primitive scheme and let $F$ be a field. Then $\rk_{\min}(F,S)=1$ if
and only if $(X,S)$ is a thin scheme of prime order.
\end{theorem}

\begin{proof}
The sufficiency is clear. To prove the necessity suppose that $\rk_{\min}(\F,S)=1$.
Then there exists a rank~$1$ matrix $A\in FS\setminus FJ$. This implies that any
two columns of~$A$ are linear dependent. So $e(A)=X^2$. On the other hand,
$e_1(A)\in S^*$ by Theorem~\ref{110707c}. Due to the primitivity of $(X,S)$ this
implies that $e_1(A)\in\{\Delta,X^2\}$. Moreover, since $A\not\in FJ$, we see
that $e_1(A)=\Delta$. Thus by formula (\ref{150807e}) we conclude that
$$
A=\sum_{x\in X}\lambda_xA_{\lambda_x}
$$
for some $\lambda_x\in F$ such that $\lambda_x\ne\lambda_y$ for all $x\ne y$. So
$e_{\lambda_x}(A)\in S$ and the valency of $e_{\lambda_x}(A)$ equals~$1$ for all
$x\in X$ (see (\ref{240707b})). This shows that the scheme $(X,S)$ is thin. To complete
the proof it suffices to note that any primitive thin scheme is of prime order.
\end{proof}

Now we can prove Theorem \ref{180707b}.

\begin{proof}[Proof of Theorem~\ref{180707b}]
From the hypothesis it follows that there exists a
rank~$r$ matrix $A\in\F_q S\setminus\F_qJ$. By Theorem~\ref{110707c} we know that
$e(A),e_1(A)\in S^*$. Since the scheme $(X,S)$ is primitive,
this implies that
$$
e(A),\ e_1(A)\in\{\Delta,X\times X\}.
$$
However, since $A$ is not a multiple of~$J$, it follows that $e_1(A)=\Delta$, and hence
\qtnl{110707d}
|X|=|\{Ax:\ x\in X\}|.
\eqtn
On the other hand, if $e(A)=X\times X$, then any two vectors $Ax$ and $Ay$ are linearly dependent.
So $r=\rk(A)=1$ and $|\{Ax:\ x\in X\}|\le|F_q^\times|=q-1$. By~(\ref{110707d}) this proves
the second part of the theorem. Thus without loss of generality we can assume that
$e(A)=\Delta$. Then any two distinct
vectors $Ax$ and $Ay$ are linearly independent. This means that $r=\rk(A)>1$ and
$$
|\{Ax:\ x\in X\}|<\frac{q^r-1}{q-1},
$$
and we are done by~(\ref{110707d}).
\end{proof}

\section{Matrix rank in the adjacency algebra}\label{180707c}
In this section we deduce Theorem~\ref{160707a} from Theorem~\ref{180707b}.
To do this, we will consider adjacency algebras over an algebraic number field
and its ring of integers. We refer to \cite{N02} for standard facts from algebraic
number theory. For the rest of the section we fix a scheme $(X,S)$,
an algebraic number field $K$ and a rational prime number~$p$.\myskip

Denote by $R$ the ring of integers of $K$. Take its prime ideal $\mathfrak{P}$ lying above~$p\Z$
and set $f$ to be the degree of~$\mathfrak{P}$. Then
\qtnl{270707a}
f\leq |K:\Q|
\eqtn
and the quotient ring $R/\mathfrak{P}$ is isomorphic to the field~$\F_q$ where $q=p^f$.
Denote by $K_\mathfrak{P}$ and $R_{\mathfrak{P}}$ the $\mathfrak{P}$-adic field and the
ring of $\mathfrak{P}$-adic integers respectively.
Then
\qtnl{270907a}
R_\mathfrak{P}=\{a\in K_\mathfrak{P}: \nu_\mathfrak{P}(a)\geq 0\}
\eqtn
where $\nu_\mathfrak{P}$ is the $\mathfrak{P}$-valuation on $K_\mathfrak{P}$. Here
$\nu_\mathfrak{P}(a)=\infty$ if and only if $a=0$. Since $R_{\mathfrak{P}}/\mathfrak{P}R_{\mathfrak{P}}\cong R/\mathfrak{P}$,
the ring epimorphism $R\to\F_q$ induces
the epimorphism $R_{\mathfrak{P}}\to\F_q, a\mapsto\ov a$, and hence the epimorphism
\qtnl{270907c}
R_\mathfrak{P}S\to\F_qS,\quad\sum_{s\in S}a_s A_s\mapsto\sum_{s\in S}\ov{a_s}A_s
\eqtn
where we use the natural identification of \{0,1\}-matrices in $R_\mathfrak{P}S$ and~$\F_qS$.
The image of $A\in R_\mathfrak{P}S$ is denoted by~$\ov A$.

\begin{lemma}\label{ha001}
Suppose $\F_qS$ is semisimple. Then every central idempotent of $K_\mathfrak{P}S$ belongs to $R_\mathfrak{P}S$.
\end{lemma}

\begin{proof}
Let $P$ be a central idempotent of $K_\mathfrak{P}S$.
Without loss of generality we assume that $P\ne 0$.
Then due to (\ref{270907a}) it suffices to verify that $\nu(P)=0$
where $\nu=\nu_\mathfrak{P}$ and given an element $A=\sum_{s\in S}a_s A_s$
of the algebra $K_\mathfrak{P}S$ we set
\qtnl{280907a}
\nu(A)=\min_{s\in S}\nu(a_s).
\eqtn
Clearly, $\nu(AB)\geq \nu(A)+\nu(B)$ and $\nu(aA)=\nu(a)+\nu(A)$
for all $A,B\in K_\mathfrak{P}S$ and $a\in K_\mathfrak{P}$. So
$$
\nu(P)=\nu(P^2)\geq\nu(P)+\nu(P)
$$
whence it follows that $\nu(P)\le 0$ (here $\nu(P)<\infty$
because $P\ne 0$). Suppose that $\nu(P)<0$. Set $Q=aP$ where $a$ is an element
of~$K_\mathfrak{P}$ such that $\nu(Q)=\nu(a)+\nu(P)=0$.
Then $\overline{Q}\ne 0$ (see (\ref{270907c})) and
$$
\nu(Q^2)=\nu(a^2P)=
\nu(a)+(\nu(a)+\nu(P))=
\nu(a)=-\nu(Q)>0.
$$
So $\overline{Q^2}=0$. Since $\overline{Q}$ is in the center of the algebra $\F_qS$, the set
$\overline{Q}(\F_qS)$ is a non-zero proper nilpotent ideal of it. However, this contradicts the
assumption that $\F_qS$ is semisimple.
\end{proof}

\begin{remark}
In the proof of Lemma~\ref{ha001} we extended the evaluation $\nu_\mathfrak{P}$ to the adjacency
algebra $K_\mathfrak{P}S$ of a scheme $(X,S)$ (see (\ref{280907a}). This extension $\nu$
has properties: $\nu(A)=\infty$ iff $A=0$, $\nu(A+B)\ge\min(\nu(A),\nu(B))$ and
$\nu(AB)\ge\nu(A)+\nu(B)$.
\end{remark}

Let $P\in\prm(\C S)$ be a central primitive idempotent of $\C S$.
Then every entry of $P$ is an algebraic number.
If the field $K$ contains all entries of $P$,
then $P\in KS$ and $K$ can be embedded into both $\C$ and $K_\mathfrak{P}$. Through
these embedding, we can regard $P$ as an element of $K_\mathfrak{P}S$.

\begin{lemma}\label{ha002}
Suppose $\F_qS$ is semisimple
and the field $K$ contains all entries of
a matrix $P\in\prm(\C S)$.
Then the following statements hold:
\nmrt
\tm{1} $P\in R_\mathfrak{P}S$; in particular, the element $\overline{P}$ is defined and belongs
to $\prm(\F_q S)$,
\tm{2}
$\overline{P}(\F_qS)\cong \Mat_n(\F_q)$ for some $n$,
the irreducible representation of $\F_qS$
defined by $\overline{P}$ is absolutely irreducible, and
the degree and the multiplicity of it in the standard representation of $\F_q S$ coincide
with $n_P$ and $m_P$ respectively (see~(\ref{190907})).
\enmrt
\end{lemma}

\begin{proof}
The first part of statement (1) immediately follows from Lemma~\ref{ha001}.
By \cite[Proposition 1.12]{Dade1973}, we can see that $\overline{P}$ is primitive.
Statement (1) is completely proved.
Next, since $P$ is primitive in $\C S$, $\overline{P}$ is primitive in
the adjacency algebra over any extension field $E$ of $\mathbb{F}_q$,
and then $\overline{P}(ES)$ is a simple algebra.
Since any finite division ring is a field,
the Wedderburn theorem shows that
$\overline{P}(\F_qS)\cong \Mat_n(F)$ for some $n$ and some finite extension $F$ of~$\F_q$,
and $\overline{P}(FS)$
is also a simple algebra.
By the separability of $F$ over $\F_q$, we have
\begin{eqnarray*}
\overline{P}(FS) &\cong&
F\otimes_{\F_q}\overline{P}(\F_qS)\cong
F\otimes_{\F_q}\Mat_n(F)\cong \Mat_n(F\otimes_{\F_q}F)\\
&\cong& \Mat_n(|F:\F_q|F)\cong |F:\F_q|\Mat_n(F).
\end{eqnarray*}
Due to the simplicity of $\overline{P}(FS)$
we have $|F:\F_q|=1$ and hence $F=\F_q$.
This means that the irreducible representation defined by $\overline{P}$
is absolutely irreducible.

Besides, the ranks of the modules
$$K_\mathfrak{P}S=P(K_\mathfrak{P}S)\oplus(I-P)(K_\mathfrak{P}S),\quad
\F_qS=\overline{P}(\F_qS)\oplus (\overline{I}-\overline{P})(\F_qS)$$
are the same. Since obviously the ranks of $P(K_\mathfrak{P}S)$ and
$(I-P)(K_\mathfrak{P}S)$ do not exceed the ranks of
$\overline{P}(\F_qS)$ and $(\overline{I}-\overline{P})(\F_qS)$, respectively,
it follows that they are equal.
Thus the degrees of irreducible representations corresponding to $P$ and
$\overline{P}$ are the same. Also comparing the dimensions of the decompositions of standard
modules $K_\mathfrak{P}X$ and $\F_qX$, we see that the multiplicities of
irreducible representations in the standard representations
corresponding to $P$ and $\overline{P}$ are the same.
\end{proof}

\begin{lemma}\label{ha003}
Suppose $\F_qS$ is semisimple and the field $K$ contains all entries of
$P\in\prm(\C S)$. Then there exists $E\in \F_qS$ such that $\rk(E)=m_P$.
\end{lemma}

\begin{proof}
From Lemma~\ref{ha002} (2), we have $\overline{P}(\F_qS)\cong \Mat_{n_P}(\F_q)$
Choose  an element $E\in\overline{P}(\F_qS)$
corresponding to a diagonal matrix unit in $\Mat_{n_P}(\F_q)$.
Since the irreducible representation corresponding to $\overline{P}$ appears
$m_P$ times in the standard representation, we have that $\rk(E)=m_P$.
\end{proof}

Now we give a proof of Theorem \ref{160707a}.

\begin{proof}[Proof of Theorem~\ref{160707a}]
Suppose $p$ is not a divisor of the Frame number $\Fr(X,S)$.
Then the adjacency algebra of $(X,S)$ over a field of characteristic $p$ is semisimple
(see Appendix).
Since the field $K=\Q(X,S)$ satisfies to condition of Lemma \ref{ha003} for all $P\in\prm(\C S)$,
one can find a matrix $E\in \F_qS$ such that $\rk(E)=m_{\min}$. Since $m_{\min}>1$, we
see that $E\not\in\F_qJ$. By Theorem~\ref{180707b} and inequality (\ref{270707a}) we have
$$|X|\leq \frac{q^{m_P}-1}{q-1}=\frac{p^{m_Pf}-1}{p^f-1}
\leq \frac{p^{|K:\Q|m_P}-1}{p^{|K:\Q|}-1}$$
and we are done.
\end{proof}

\section*{Appendix : Association schemes}

Let $X$ be a finite set and $S$ a partition of~$X^2$ closed with respect to the transpose.
A pair $(X,S)$ is called an {\it associative scheme} or {\it scheme} if
the reflexive relation $\Delta$ belongs to the set $S$ and given $r,s,t\in S$, the number
\qtnl{240707a}
c_{r,s}^t=|\{z\in X:\,(x,z)\in r,\ (z,y)\in s\}|
\eqtn
does not depend on the choice of $(x,y)\in t$. The elements of $S$ and the number $|X|$ are
called the {\it basis relations} and the {\it order} of the scheme.
The set of unions of all subsets of $S$ is denoted by~$S^*$.
The number $d_r=c_{r,r^*}^\Delta$ where $r^*$ is the transpose of $r$, is called the {\it valency}
of~$r$. The scheme $(X,S)$ of order $\ge 2$ is called {\it primitive} if any equivalence relation on~$X$
belonging to $S^*$ coincides with either $\Delta$ or~$X^2$.\myskip

Given a field $F$ the linear span $FS$ of the set $\{A_s:\ s\in S\}$ forms a subalgebra
of the algebra~$\Mat_X(F)$ (see (\ref{240707a})). This subalgebra is called the {\it adjacency
algebra} of the scheme $(X,S)$ over $F$.
From the definition it follows that $FS$ is closed with respect
to the transpose and the Hadamard multiplication. In particular,
\qtnl{240707b}
a\in F,\ A\in F S\ \Rightarrow\ A^{(a)}\in F S
\eqtn
where $A^{(a)}$ is a \{0,1\}-matrix in $\Mat_X(F)$ such that $A^{(a)}_{x,y}=1$ if and only
if $A_{x,y}=a$. One can see that any \{0,1\}-matrix belonging to $FS$ is of the form
$A_s$ for some $s\in S^*$. The set of all central primitive idempotents of the algebra $FS$
is denoted by $\prm(FS)$.\myskip

The adjacency algebra~$\C S$ of the scheme $(X,S)$ over the complex number field $\C$
is semisimple. So by the Wedderburn
theorem its {\it standard module} $\C X$ is completely reducible. For an
irreducible submodule $L$ of $\C X$ corresponding to a central primitive idempotent~$P$ of the
algebra~$\C S$, we set
\qtnl{190907}
n_P=\dim_\C(L),\quad m_P=\rk(P)/n_P,
\eqtn
thus $m_P$ and $n_P$ are the
{\em multiplicity} and the {\em degree} of the corresponding irreducible representation of~$\C S$.
It is known that $m_P\geq n_P$ for all~$P$ \cite{EP97}.
Obviously, for the {\it principal} central primitive idempotent $P=(1/|X|)J$ of the algebra~$\C S$
we have $m_P=n_P=1$.\myskip

For an arbitrary field $F$,
the semisimplicity of the algebra $FS$ was studied in~\cite{H00}. It was proved that it
is semisimple if and only if the characteristic of the field $F$ does not divide the number
$$
\Fr(X,S)=|X|^{|S|}\frac{\prod_{r\in S}d_r}{\prod_{P\in\prm}m_P^{n_P^2}}
$$
where $\prm$ is the set of all non-principal central primitive idempotents
of the algebra~$\C S$. This number is
called the {\it Frame number} of the scheme~$(X,S)$.\myskip

A scheme $(X,S)$ is called {\it thin}, if $d_r=1$ for all $r\in S$. In this case there exists
a regular group $G\le\sym(X)$ such that $S$ coincides with the 2-orbits of~$G$, i.e. the
orbits of the componentwise action of~$G$ on the set~$X^2$. (In this case the sets $X$ and $G$ can
be naturally identified and the algebra $F S$ becomes the group algebra $F G$.) Exactly the
same construction produces a scheme $(X,S)$ for an arbitrary transitive group $G\le\sym(X)$.
One can prove that such a scheme is primitive if and only if the group $G$ is primitive.

\end{document}